\documentclass[11pt]{amsart}
\usepackage{graphicx}
\usepackage{amsmath,amsthm}
\usepackage{amssymb}
\usepackage[left = 3cm, right = 3 cm , top = 3 cm , bottom = 3cm ]{geometry}
\usepackage[utf8]{inputenc}
\usepackage{pinlabel}

\newtheorem{thm}{Theorem}[section]
\newtheorem{prop}[thm]{Proposition}
\newtheorem{lem}[thm]{Lemma}

\newtheorem{conj}[thm]{Conjecture}

\newtheorem{question}[thm]{Question}

\newcommand{\Text}[1]{\text{\textnormal{#1}}}

\theoremstyle{remark}
\newtheorem{remark}[thm]{Remark}

\theoremstyle{definition}

\begin{document}

\title{Lower Bound for Dilatations \footnote{Partially supported by NSF Grants DMS-1006553 and DMS-1607374.}}
\author{Mehdi Yazdi }
\date{}

\begin{abstract} We prove a new lower bound for the dilatation of an arbitrary pseudo-Anosov map on a surface of genus $g$ with $n$ punctures. Our bound improves the former super-exponential dependence on the genus by a polynomial dependence.
\end{abstract}

\maketitle

\section{Introduction}
Let $S = S_{g,n}$ be a surface of genus $g$ with $n$ punctures, where $\chi(S)=2-2g-n<0$. The mapping class group of $S$, $Mod(S_{g,n})$, is the group of orientation preserving homeomorphisms of $S$ up to isotopy. Here, the punctures are assumed to be fixed setwise by the homeomorphism. Nielsen-Thurston classification of mapping classes states that every mapping class is either pseudo-Anosov, reducible or finite order \cite{thurston1988geometry}. Pseudo-Anosovs are often the ones whose understanding is the crucial part in studying the mapping class group. 

Associated to any pseudo-Anosov map is an algebraic integer called the dilatation or the stretch factor. The dilatation measures how much the map stretches/shrinks in the two canonical directions at each point of the surface. From a dynamical point of view, the logarithm of the dilatation is the entropy of the pseudo-Anosov map. Ivanov proved that on a fixed surface, the set of dilatations is a discrete subset of $(1,\infty)$ \cite{ivanov1990stretching,arnoux1981construction}. In particular there exists a minimum dilatation. Let us denote by $l_{g,n}$ the logarithm of the minimum dilatation for pseudo-Anosov maps on $S_{g,n}$. Finding the minimum dilatation or its asymptotic behavior has been of great importance. One motivation is that $l_{g,n}$ is the systole (the length of the shortest geodesic) of the moduli space with the Teichm\"{u}ller metric. Another motivation comes from the relation between low-dilatation pseudo-Anosov maps and low-volume fibered hyperbolic 3-manifolds \cite{aaber2010closed}. Penner found the asymptotic behavior of this number for closed surfaces \cite{penner1991bounds}. He proved that there are constants $c_1 , c_2 >0$ such that for any $g\geq 2$

\[ \frac{c_1}{g} \leq l_{g,0} \leq \frac{c_2}{g}. \]

\noindent Our aim is to understand the asymptotic behavior of $l_{g,n}$ similarly. Recall that Penner has proved the following \cite{penner1991bounds}
\[ l_{g,n} \geq \dfrac{\log(2)}{12g-12+4n}. \]
which is comparable to $\dfrac{1}{|\chi(S)|}$, up to multiplicative constants. 
Tsai has obtained another lower bound for $l_{g,n}$, which gives a better bound than Penner's theorem when $n$ is large compared to $g$ \cite{tsai2009asymptotic}. Let $\Gamma_S(3)$ denote the kernel of the action of $Mod(S_{g,0})$ on $H_1(S_{g,0}\medskip;\smallskip\mathbb{Z}/3\mathbb{Z})$. Define 
\[ \Theta(g):=[Mod(S):\Gamma_S(3)]. \]
Note that $\Theta(g)$ is super-exponentially large in $g$ \cite{tsai2009asymptotic}\footnote{In fact standard theorems imply that it is larger than $3^{g^2}$. See the background section.}.

\begin{thm}(Tsai)
For any $g \geq 2$ and $n \geq 0$ we have the following:
\[ l_{g,n} \geq  \min \left\{ \frac{1}{\Theta(g)}\frac{\log(2)}{(12g-12)}, \frac{1}{\Theta(g)} \frac{\log(3|\chi(S)|)}{6|\chi(S)|} \right\}.\]
\label{tsai}
\end{thm}
\noindent Note that when $n$ is large compared to $g$, the minimum is the second expression. The following theorem shows that one can replace $\Theta(g)$ in Tsai's theorem by a term that is polynomially small in $g$.

\begin{thm}
Given any positive real number $\alpha$, there exists a positive constant $C = C(\alpha)$ such that for any $g \geq 2$ and $n \geq 0$ we have the following:
\[ l_{g,n} \geq \frac{C}{g^{2+\alpha}}\frac{\log(|\chi(S)|)}{|\chi(S)|}. \]
\label{maintheorem}
\end{thm}

\noindent Our lower bound should be compared with Tsai's upper bound for $l_{g,n}$ \cite{tsai2009asymptotic}. Tsai proved that there is a constant $C>0$ such that for any $g \geq 2$ and $n \geq 0$ the following holds \footnote{Tsai proved this bound for $n \geq 12g+7$. See the appendix for an extension of her result to all $n \geq 0$.}: 
\[ \l_{g,n} \leq C \medskip g \medskip \frac{\log|\chi(S)|}{|\chi(S)|}. \]

Here is the idea of the proof. Following Tsai, we look at the Lefschetz number of the map $f : S \longrightarrow S$. If the Lefschetz number of $f$ is negative then one can give a 'good' lower bound for the dilatation. However, the Lefschetz number of a pseudo-Anosov map need not to be negative in general. We prove that there is a 'relatively small' number (at most polynomially large in genus), $k$, such that the Lefschetz number of $f^k$ is negative. Using the Lefschetz formula for the Lefschetz number of a map, this translates into a problem about traces of powers of integral matrices. Then we use elementary Fourier analysis and Dobrowolsky's theorem about modulus of algebraic integers to prove the desired statement. 
\subsection{Acknowledgement}
This work has been done during my PhD studies at Princeton University. I would like to thank my advisor David Gabai for his constant support and encouragement. Special thanks to Peter Sarnak for helpful discussions on Turan theory and suggesting the reference \cite{montgomery1994ten}. I would like to thank Ian Agol and Will Sawin for helpful comments and Bal\'{a}zs Strenner and Masoud Zargar for reading an earlier version of this paper. 

\section{Background}
\noindent Throughout, we assume that the surface $S$ is orientable and $\chi(S) = 2-2g-n<0$.
\subsection{Thurston-Nielsen Theory}\hfill\\

Thurston-Nielsen classification of mapping class group states that each element in the mapping class group can be represented by a map $f$ that is one of the following:\\
1) periodic \\
2) reducible \\
3) pseudo-Anosov\\
\textbf{Periodic} means that $f$ has a power that is equal to identity. \textbf{Reducible} means that there is a collection $\mathcal{C}$ of disjoint simple closed curves on $S$ that is preserved by $f$, i.e., $f(\mathcal{C})=\mathcal{C}$. \textbf{Pseudo-Anosov} means that there is a pair of transverse measured foliations $\mathcal{F}^\pm$ on $S$ and a positive number $\lambda > 0$ such that the foliations are preserved by $f$ but their measures are expanded/contracted by a factor of $\lambda > 1$, i.e., $f(\mathcal{F}^+)=\lambda \mathcal{F}^+$ and  $f(\mathcal{F}^-)=\frac{1}{\lambda} \mathcal{F}^-$. The foliations $\mathcal{F}^\pm$ might have prong-type singularities (Figure \ref{prongs}). The number $\lambda$ is called the \textbf{dilatation} or \textbf{stretch factor} of $f$.\\
\begin{figure}
\centering
\includegraphics[width= 3 in]{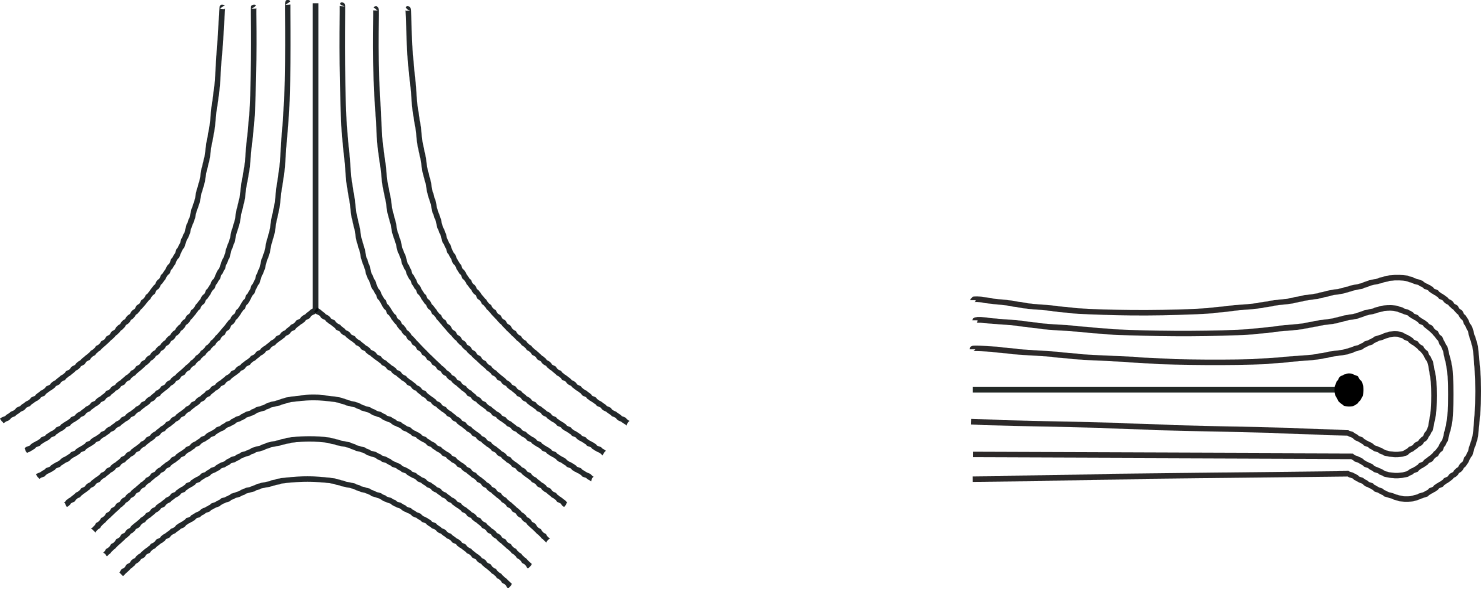}
\caption{Left: A 3-prong singularity, Right: We allow 1-prong singularities around the punctures.}
\label{prongs}
\end{figure}

\subsection{Previous bounds for dilatations}\hfill\\

Penner originated the study of minimal dilatations for orientable surfaces. He proved that $l_{g,0}$ behaves asymptotically like $\frac{1}{g}$. He also gave a lower bound of the order $\frac{1}{|\chi(S)|}$ for the value of $l_{g,n}$ \cite{penner1991bounds}. Since then, there has been a lot of effort for understanding the minimum stretch factor from at least two different perspectives. \\
The first one tries to make the constants in Penner's original theorem sharp, for small values of $g$ or asymptotically. McMullen's question is in this direction \cite{mcmullen2000polynomial}. 
\begin{question}(McMullen)
Does $\lim_{g \rightarrow \infty} \hspace{2mm} g \hspace{1mm}. \hspace{1mm} l _{g,0}$ exists? What is its value? 
\end{question}
\noindent There has been a lot of progress in finding upper bounds for $g \hspace{1mm}. \hspace{1mm} l _{g,0}$ \cite{bauer1992upper,minakawa2006examples}. The lower bound seems to be much more difficult (see the work of McMullen \cite{mcmullen2015entropy}). \\
The second direction seeks for understanding the behavior of $l_{g,n}$ along different subsets of the $(g,n)$ plane, at least up to multiplicative constants. Theorem \ref{tsai} is of this form. It implies that the behavior of $l_{g,n}$ along the line $\underline{ g = \Text{Constant}}$ is like $\frac{\ln(n)}{n}$ when $g \geq 2$. Valdivia showed for any fixed $r \in \mathbb{Q}^+$, the behavior of $l_{g,n}$ along the line $\underline{ g = rn }$ is like $\frac{1}{g}$ (which is the same behavior as $\frac{1}{n}$ in this case) \cite{valdivia2012sequences}, i.e., for any $r \in \mathbb{Q}$ theres are constants $D_1 = D_1(r)$ and $D_2= D_2(r)$ such that for any $n \in \mathbb{N}$ and $g =rn$ we have
\[ \frac{D_1}{g} \leq l_{g,n} \leq \frac{D_2}{g}\]
It is tempting to understand the behavior of $l_{g,n}$ as a two variable function.
\begin{question}
What is the behavior of $l_{g,n}$ as a function of two variables in the $(g,n)$ plane?
\end{question}
\subsection{Markov Partition}\hfill\\
Let $f : S \longrightarrow S$ be a pseudo-Anosov map with invariant measured foliations $\mathcal{F}^+$ and $\mathcal{F}^-$. A rectangle is a map $\phi : I \times I \longrightarrow S$ such that $\phi$ is an embedding when restricted to the interior of $I \times I$. Moreover, $\phi( \Text{point} \times I) \subset \mathcal{F}^+$ and $\phi(I \times \Text{point}) \subset \mathcal{F}^-$. Define the $\pm$ boundary of $\phi$ as $\partial_+ = \phi( \partial I \times I)$ and $\partial_- = \phi(I \times \partial I)$ (Figure \ref{rectangle}). We usually do not distinguish between a rectangle and its image, $\mathcal{R}$, by abuse of notation.
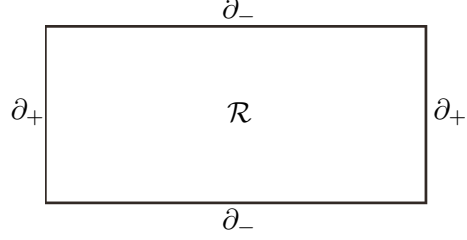
\begin{figure}
\labellist
\pinlabel $\mathcal{R}$ at 115 55
\pinlabel $\partial_+$ at -10 55
\pinlabel $\partial _+$ at 240 55
\pinlabel $\partial_-$ at 115 115
\pinlabel $\partial_-$ at 115 -10
\endlabellist
\centering 
\includegraphics[width= 2 in]{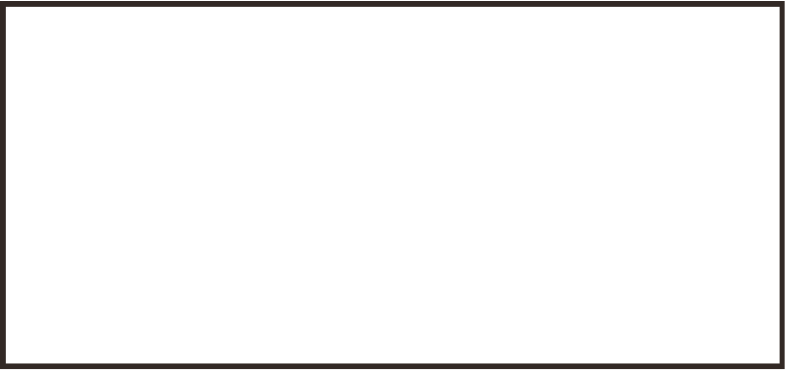}
\caption{A rectangle}
\label{rectangle}
\end{figure}
\noindent A \textbf{Markov partition for $f$} is a finite family of rectangles $\{ \mathcal{R}_i \}$ that cover the whole surface and satisfy the following three conditions.\\
i) The rectangles do not intersect in the interior. \\
ii) For each rectangle $\mathcal{R}_i$, $f(\partial_+ \mathcal{R}_i) \subset \bigcup \limits_{j} \partial_+ \mathcal{R}_j $.\\
iii) For each rectangle $\mathcal{R}_i$, $f^{-1}(\partial_- \mathcal{R}_i) \subset \bigcup \limits_{j} \partial_- \mathcal{R}_j $.\\
Any pseudo-Anosov map has a Markov partition. Bestvina-Handel have constructed a Markov partition of size at most $9 |\chi(S)|$ for $f$ when the surface is closed and a Markov partition of size at most $3|\chi(S)|$ when the surface has at least one marked point. Define the transition matrix $A=(a_{i,j})$ associated to the Markov partition as follows. The entry $a_{i,j}$ counts the number of times that $f(\mathcal{R}_i)$ wraps around $\mathcal{R}_j$. Bestvina-Handel showed that this matrix can be chosen to be Perron-Frobenius. Moreover, its maximal eigenvalue is equal to the dilatation of $f$. In particular, $\lambda(f)$ is an algebraic integer.
\subsection{Lefschetz number}\hfill\\

Let $M$ be a compact, oriented manifold and $f : M \longrightarrow M$ be a map. The \textbf{Lefschetz number of $f$}, $L(f)$, is defined as the algebraic intersection of the graph of $f$ and the diagonal inside $M \times M$. Therefore it is invariant under homotopy of the map $f$. The \textbf{Lefschetz formula} states that this number can be computed in two different ways. On one hand, it is equal to the following sum coming from the action of $f$ on homology groups of $M$:
\[ \sum_{i\geq0} (-1)^i \hspace{2mm} Tr(f_* : H_i (M;\mathbb{R}) \longrightarrow H_i(M;\mathbb{R})) \]
On the other hand when $f$ has isolated fixed points, the Lefschetz number of $f$ is equal to sum of the local Lefschez numbers at fixed points. If $p$ is an isolated fixed point of $f$, then the local Lefschetz number of $f$ at $p$, $L_p(f)$, is defined as follows. Take a small sphere, $U$, around $p$ that contains no other fixed point. Then $L_p(f)$ is equal to the degree of the map $z \mapsto \frac{f(z)-z}{|f(z)-z|}$ restricted to $U$.

\[ L(f) = \sum_{f(p)=p} \hspace{2mm} L_p(f) = \sum_{i\geq0} (-1)^i \hspace{2mm} Tr(f_* : H_i (M;\mathbb{R}) \longrightarrow H_i(M;\mathbb{R})). \]
Note that when $M=S$ is a compact orientable surface, the above formula simplifies to the following:
\[ L(f) =\sum_{f(p)=p} \hspace{2mm} L_p(f)= 2 - Tr(f_* : H_1(S) \longrightarrow H_1(S)). \]
The following crucial observation is due to Tsai. We bring the proof from \cite{tsai2009asymptotic} for the reader's convenience. 

\begin{lem}(Tsai) Let $f : S \longrightarrow S$ be a pseudo-Anosov map on a surface with at least one marked point. Assume that $L(f) < 0$. We have the following estimate for the stretch factor of $f$. 
\[ \log(\lambda(f)) \geq \frac{\log(3|\chi(S)|)}{6|\chi(S)|}.  \]
\label{negative-lefschetz}
\end{lem}
\begin{proof}
\textbf{First Step}: There exists a Markov partition and a rectangle $\mathcal{R}$ of the partition such that the interior of $\mathcal{R}$ and $f(\mathcal{R})$ intersect. \\

The map $f$ has a Markov partition with $k$ rectangles where $k \leq 3|\chi(S)|$ \cite{bestvina1995train}. The map $f$ has isolated singularities. Since the Lefschetz number of $f$ is negative, at least one of the local Lefschetz numbers of $f$, say at $p$, should be negative. We show that one of the rectangles that contain $p$ (in the interior or on the boundary) is the rectangle that we are looking for. If there exist a rectangle $\mathcal{R}$ that contains $p$ in the interior then we are done since $p \in \mathcal{R} \cap f(\mathcal{R})$. Otherwise, we claim that $p$ has to be of the following types:\\

i) $p$ is a non-singular fixed point and the transverse orientation of $\mathcal{F}^+$ at $p$ is preserved,\\

ii) $p$ is a singular fixed point and at least one of the separatrices emanating from $p$ is fixed by $f$,\\

Note that it is clear that if $p$ is of the above types then one of the rectangles of the Markov partition (constructed by Bestvina-Handel \cite{bestvina1995train}) around $p$ has the desired property. If $p$ is not of the above type then it would be one of the following:\\

iii) $p$ is a non-singular fixed point and the transverse orientation of $\mathcal{F}^+$ at $p$ is reversed, or\\

iv) $p$ is a singular fixed point and none of the separatrices emanating from $p$ are fixed by $f$.\\

However, direct calculation shows that in the third and fourth case, the local Lefschetz number of $p$ is equal to $+1$ which is inconsistent with our assumption about $p$ (see page 2262 of \cite{tsai2009asymptotic}). This completes the proof. 
\\

\textbf{Second Step}: As a corollary of the first step, the transition matrix associated to the Markov partition, $A=(a_{i,j})$, has a nonzero entry on the diagonal. Define an oriented graph $G$ with the vertex set $V$ such that there are $a_{i,j}$ oriented edges from $v_i$ to $v_j$. Since $A$ is Perron-Frobenius, $G$ is path connected by oriented paths. By the previous step, there is an $\ell$ such that $a_{\ell,\ell}>0$, therefore there is at least one edge from $v_\ell$ to itself. On the other hand any two vertices of $G$ are connected by an oriented path of length at most $k$. Hence, for any $i,j$ there are oriented paths of length at most $k$ from $v_i$ to $v_\ell$ and from $v_\ell$ to $v_j$. Putting these paths together and adding the loop at the vertex $v_\ell$ as much as necessary, we get an oriented path of length $2k$ from $v_i$ to $v_j$. We just showed that all entries of the matrix $A^{2k}$ are positive. This means that $\mu(A^{2k}) \geq k$ since the spectral radius of a non-negative matrix is bounded below by the minimum row (column) sum. Hence
\[ \log( \mu(A)) \geq \frac{\log(k)}{2k} \geq \frac{\log(3|\chi(S)|)}{6|\chi(S)|}.  \]
Note that by the Lefschetz formula, when $L(f) < 0$ the map $f$ has at least one fixed point and we can take the fixed point as a marked point. In other words, the condition of having at least one marked point is redundant here.
\end{proof}

\subsection{The order of $\Theta(g)$}\hfill\\

In this part, we briefly explain why $\Theta(g) > 3 ^{g^2}$. Let $\Gamma_S(3)$ be the kernel of the composition
\[ Mod(S_g) \longrightarrow Sp(2g,\mathbb{Z}) \longrightarrow Sp(2g,\mathbb{F}_3) \] 
where $\mathbb{F}_3$ is the field of three elements. Both of these maps are surjective. The first one is well known in the mapping class group theory (see for example \cite{farb2011primer}). The second one follows from strong approximation. (Morally speaking, it says that the mod $q$ solutions to a system of quadratic equations can be lifted to integral solutions under suitable conditions, where $q$ is a prime number \cite{kneser1966strong}.) Therefore, the index of the kernel is equal to the order of the image:
\[ \Theta(g) =  [Mod(S):\Gamma_S(3)] = |Sp(2g,\mathbb{F}_3)|. \]
But the order of $Sp(2m,\mathbb{F}_q)$ over a finite field $\mathbb{F}_q$ with $q$ elements is equal to \cite{grove2002classical}:
\[ q^{m^2} \prod_{i=1}^{m} \hspace{2mm} (q^{2i}-1) \]
which is obviously greater than $q^{m^2}$.

\subsection{On the modulus of algebraic integers}\hfill\\

In the proof of Proposition \ref{main}, we use some facts about modulus of algebraic integers. A complex number $\lambda$ is an \textbf{algebraic integer}, if it is a root of a monic polynomial with integer coefficients. The \textbf{degree of $\lambda$}, is the smallest possible degree of such polynomial. The smallest degree polynomial is called the $\textbf{minimal polynomial of $\lambda$}$. The \textbf{Galois conjugates of $\lambda$} are all the roots of the minimal polynomial, including $\lambda$ itself. Define $\overline{|\lambda|}$ to be the maximum modulus amongst all Galois conjugates of $\lambda$. Clearly, $\overline{|\lambda|} \geq 1$ and equality happens for roots of unity. \textbf{Kronecker's theorem} states that if $\overline{|\lambda|} = 1$ then $\lambda$ is a root of unity. Moreover, if $\lambda $ is not a root of unity and has degree $d$, then $\overline{|\lambda|} - 1$ is bounded below by a number that just depends on $d$. The conjectural best bound is of order $\frac{1}{d}$. In fact by looking at the number $\lambda = 2^{\frac{1}{d}}$, it is easy to see that this is the best one can hope for. This is called the Schinzel-Zassenhaus conjecture \cite{schinzel1965refinement}.

 \begin{conj}(Schinzel-Zassenhaus)
There exists a constant $c>0$ such that for any algebraic integer $\lambda \neq 0$ of degree $d$ which is not a root of unity we have
\[ \overline{|\lambda|} \geq 1+\frac{c}{d}. \]
\label{conj:Schinzel}
 \end{conj}
\noindent Although Schinzel-Zassenhaus conjecture is still open, a slightly weaker form of it has been proved by Dobrowolsky \cite{dobrowolski1979question}.

\begin{thm}(Dobrowolsky)
Let $\lambda$ be an algebraic integer of degree $d$. For large enough $d$ if $\lambda$ is not a root of unity then
\[ \overline{|\lambda|} \geq 1+\frac{1}{d} \bigg( \frac{\log\log(d)}{\log(d)} \bigg)^3. \]
\label{thm:Dobrowolsky}
\end{thm}
\noindent Note that Dobrowolsky theorem does not take care of small values of $d$. Therefore we use the following theorem of Schinzel-Zassenhaus for small values of $d$ \cite{schinzel1965refinement}. 

\begin{thm}(Schinzel-Zassenhaus)
If an algebraic integer $\lambda \neq 0$ is not a root of unity, and if $2s$ among its conjugates have nonzero imaginary part, then
\[ \overline{|\lambda|} > 1+4^{-s-2}. \]
\label{thm:Schinzel}
\end{thm}

\section{Proof of Theorem \ref{maintheorem}}

\begin{proof}
Let $f \in Mod(S_{g,n})$ be a pseudo-Anosov map. Denote by $\lambda(f)$ the dilatation of $f$. The idea is to look at the Lefschetz number of $f$, which we denote by $L(f)$. Define $\hat{f} \in Mod(S_{g,0})$ to be the map obtained by forgetting the punctures. The following two observations have been made by Tsai \cite{tsai2009asymptotic}.

1) $L(f) = L(\hat{f})$. 

2) If $L(f)<0$ and $f$ is pseudo-Anosov, then $\log(\lambda(f)) \geq \frac{\log(3|\chi(S)|)}{6|\chi(S)|}$ (see Lemma \ref{negative-lefschetz}).\\
The aim is to find a suitable power $\nu$ of $f$ such that $L(f^{\nu})<0$ and then use the above bound. For any map $\phi \in Mod(S_{g,0})$, we have $L(\phi)=2-Tr(\phi _*)$ where $\phi_* : H_1(S) \longrightarrow H_1(S)$ is the induced map on homology. Proposition \ref{main} shows that one can find such a power that is at most polynomially large in terms of the genus.\\
In Proposition \ref{main}, take $B=2$, $\epsilon = \alpha$, $\phi = \hat{f}$ and $A=\phi_*$. Therefore $m=2g$. Hence, if $g \gg 0$ there is some $\nu \leq (2g)^{2+\alpha}$ such that 
\[ L(f^{\nu})=L\left(\widehat{f^{\nu}}\right)=L\left((\hat{f})^{\nu}\right)=2-Tr(A^\nu)<0. \]
Since $f^{\nu}$ is pseudo-Anosov we have the following
\[\nu \log(\lambda(f)) = \log(\lambda(f^{\nu})) \geq \frac{\log(3|\chi(S)|)}{6|\chi(S)|}. \]
\[ \Rightarrow  \log(\lambda(f)) \geq \frac{1}{\nu} \frac{\log(3|\chi(S)|)}{6|\chi(S)|} \geq \frac{1}{(2g)^{2+\alpha}}  \frac{\log(3|\chi(S)|)}{6|\chi(S)|}.  \]
This finishes the proof when $g \gg 0$ let say for $g \geq N$. For the finitely many remaining values of $2 \leq g<N$, we use Lemma \ref{Dirichlet} \cite{montgomery1994ten}. Since $det(A)=1$, by Lemma \ref{Dirichlet} there exist a $1 \leq \nu \leq 8^{2g}$ such that
\[ Tr(A^{\nu}) \geq \frac{2g}{\sqrt{2}} > 2. \]
Therefore
\[ \log(\lambda(f)) \geq \frac{1}{\nu} \frac{\log(3|\chi(S)|)}{6|\chi(S)|} \geq \frac{1}{8^{2g}} \frac{\log(3|\chi(S)|)}{6|\chi(S)|}. \]
So if we define $c_j = \dfrac{\hspace{1mm}j^{2+ \alpha}\hspace{1mm}}{8^{2j}}$ and set
\[  C' = \min \left\{ c_1 , ... , c_{N-1}, \frac{1}{2^{2+\alpha}}\right\}. \]
Then, we have the following for each $g \geq 2$ and $n \geq 0$
\[ \log(\lambda(f)) \geq \frac{C'}{g^{2+\alpha}} \frac{\log(3|\chi(S)|)}{6|\chi(S)|} \geq \frac{C}{g^{2+\alpha}} \frac{\log(|\chi(S)|)}{|\chi(S)|} \]
for $C = \dfrac{C'}{6}$. \\
\end{proof}

\begin{remark}
One can use Theorem \ref{tsai} instead of Lemma \ref{Dirichlet} to take care of the finitely many remaining values of $g<N$. However, we preferred to use a more elementary approach.
\end{remark}

\noindent The next Lemma has been used in the proof of Theorem \ref{maintheorem}.  

\begin{lem}
Let $z_1 , ... , z_m$ be complex numbers. Define 
\[ S_{\nu} = z_1^{\nu}+... + z_m^{\nu}. \]
There is a $\nu$, $1 \leq \nu \leq 8^m$ such that 
\[ Re(S_{\nu}) \geq \frac{1}{\sqrt{2}} \sum_{j=1}^{m}|z_j|^{\nu}. \]
In particular if $|z_1 ... z_m| =1$ then there is a $\nu$, $1 \leq \nu \leq 8^m$ such that $Re(S_{\nu}) \geq \frac{m}{\sqrt{2}}$.
\label{Dirichlet}
\end{lem}

\begin{proof}
Decompose the plane into $8$ equal sections according to the angle. For $1 \leq i \leq 8$:
\[ V_i = \{ (r,\theta) \in \mathbb{R}^2 | (i-1)\frac{2\pi}{8} \leq \theta < i \frac{2\pi}{8} \}. \]
For each $1 \leq k \leq 8^m+1$ we code the regions in which the points $z_1^k, ... , z_m^k$ lie with a vector 
\[A_k = (a_1 , ... , a_m) \]
where $1 \leq a_i \leq 8$. By the pigeonhole principle there are distinct indices $1 \leq i,j \leq 8^m+1$ such that $A_{i} = A_j$. Therefore
\[ A_{|j-i|} = (b_1 , ... , b_m)\]
where $b_{\ell} \in \{ 1,8\} $ for each $1 \leq \ell \leq m$. This implies that for $\nu = |j-i|$
\[ Re(S_{\nu}) \geq \frac{1}{\sqrt{2}} \sum_{j=1}^{m}|z_j|^{\nu}. \]
The conclusion of the second part of the lemma is obtained by using the AM-GM inequality:
\[\frac{ |z_1|^{\nu}+...+|z_m|^{\nu}}{m} \geq \sqrt[m]{|z_1 ... z_m|^{\nu}}=1. \] 
\end{proof}

\noindent For any real matrix $A$, we use the notation $\rho(A)$ for the spectral radius of $A$, i.e. the largest absolute value of its eigenvalues. The next Proposition is the main technical result that has been used in this paper.
\begin{prop}
Fix $B>0$ and $\epsilon >0$. There exist $n=n(B,\epsilon)$ such that for any $m \geq n$ and any $A \in SL(m,\mathbb{Z})$ we have the following:\\
There is some $\nu$, $1 \leq \nu \leq m^{2+\epsilon}$ such that 
\[Tr(A^\nu)>B. \]
\label{main}
\end{prop}

\noindent Proposition \ref{main} obviously follows from the combination of Propositions \ref{case1} and \ref{case2}.

\begin{prop}
Fix $B>0$ and $\epsilon >0$. There exist $n=n(B,\epsilon)$ such that for any $m \geq n$ and any $A \in SL(m,\mathbb{Z})$ with $\rho(A)>1$ we have the following:\\
There is some $\nu$, $1 \leq \nu \leq m^{1+\epsilon}$ such that 
\[Tr(A^{\nu})>B. \]
\label{case1}
\end{prop}

\begin{prop}
Fix $B>0$ and $\epsilon >0$. There exist $n=n(B,\epsilon)$ such that for any $m \geq n$ and any $A \in SL(m,\mathbb{Z})$ with $\rho(A)=1$ we have the following:\\
There is some $\nu$, $1 \leq \nu \leq m^{2+\epsilon}$ such that 
\[Tr(A^\nu)>B. \]
\label{case2}
\end{prop}

\noindent \textbf{Proof of Proposition \ref{case1}}

\begin{proof}
Recall the following theorems of Dobrowolsky \cite{dobrowolski1979question} and Schinzel-Zassenhaus \cite{schinzel1965refinement}
\newtheorem*{thm:Dobrowolsky}{Theorem \ref{thm:Dobrowolsky}}
\begin{thm:Dobrowolsky}(Dobrowolsky)
Let $\lambda$ be an algebraic integer of degree $d$ and define $\overline{|\lambda|}$ to be the maximum modulus between all Galois conjugates of $\lambda$, including itself. For large enough $d$ if $\lambda$ is not a root of unity then
\[ \overline{|\lambda|} \geq 1+\frac{1}{d} \bigg( \frac{\log\log(d)}{\log(d)} \bigg)^3. \]
\end{thm:Dobrowolsky}

\newtheorem*{thm:Schinzel}{Theorem \ref{thm:Schinzel}}
\begin{thm:Schinzel}(Schinzel-Zassenhaus)
If an algebraic integer $\lambda \neq 0$ is not a root of unity, and if $2s$ among its conjugates have nonzero imaginary part, then
\[ \overline{|\lambda|} > 1+4^{-s-2}. \]
\end{thm:Schinzel}

\noindent The two theorems together imply that there exist a constant $c>0$ such that for all $d$
\[\overline{|\lambda|} \geq 1 + \frac{c}{d\smallskip\log(d)^3} \hspace{10mm}(*) \]
This is because by Dobrowolsky's theorem one can take $c=1$ for large $d$, say for $d \geq M$. For the finitely many remaining values of $2 \leq d < M$ one can take $c = 4^{-M-2}$. Hence, in general $c= \min\left\{1, 4^{-M-2}\right\}$ works.\\
Let $\lambda_1, ... , \lambda_m$ be the eigenvalues of $A$ with $\lambda_1$ having the maximum modulus between them. Therefore $|\lambda_1| > 1$. By the previous discussion we have the following:
 \[ |\lambda_1| \geq 1 + \frac{c}{d\smallskip\log(d)^3} \geq  1 + \frac{c}{m\log(m)^3} \smallskip.\]
Define $z_j = \frac{\lambda_j}{|\lambda_1|}$. Hence $z_1 , ... , z_m$ are complex numbers with $\max |z_j|=1$. Define 
\[ S_{\nu} = z_1^{\nu} +...+ z_m^{\nu}. \]
In particular $S_{\nu}$ is always a real number by Newton identities. Set $K_0 = 20(\dfrac{B}{c}) \left(m \log(m)^3\right)$, where $c$ is the constant in $(*)$. Set $K = m^{1+ \epsilon}$. Note that for $m \gg 0$ we have $K \geq 5(m+2BK_0)$. We consider two cases\\
 
\noindent 1) There exists $1 \leq \nu \leq K_0 $ such that $S_{\nu} > B$. Then 
\[ Tr(A^{\nu}) = |\lambda_1|^{\nu}\smallskip S_{\nu} \geq S_{\nu} >B. \]

\noindent 2) For each $1 \leq \nu \leq K_0$ we have $S_{\nu} \leq B$. The proof in this case follows the lines of the proof of Cassel's theorem \cite{montgomery1994ten}. Let $P(z)= \frac{1}{2}+ \sum_{\nu=1}^{K} (1- \frac{\nu}{K+1})z^{\nu}$. Then $Re(P(z)) \geq 0$ whenever $|z| \leq 1$ by the properties of the Fejer kernel. Let $z_j = r_j e(\theta _j) := r_j e^{2\pi i \theta_j}$. We have the following
 
 \[ \sum_{\nu = 1}^{K} (1-\frac{\nu}{K+1})(1+\cos 2\pi \nu \theta_1)Re(S_{\nu}) =
 \sum_{j=1}^{m} \sum_{\nu=1}^{K}(1-\frac{\nu}{K+1})r_j^{\nu}(1+\cos 2\pi \nu \theta_1)\cos 2 \pi \nu \theta_j 
  \]
  \[ = \sum_{j=1}^{m} Re[P(z_j)+ \frac{1}{2}P(r_j e(\theta_j - \theta _1))+ \frac{1}{2}P(r_j e(\theta_j + \theta_1))-1] \]
  Since $P(r_1)=P(1)=\frac{K+1}{2}$, we obtain that the above is 
  \[ \geq \frac{K+1}{4} - m .\]
  Now we have the following estimate
  \[ \sum_{\nu = K_0}^{K} (1-\frac{\nu}{K+1})(1+\cos 2\pi \nu \theta_1)Re(S_{\nu}) =  \]
  \[ \sum_{\nu = 1}^{K} (1-\frac{\nu}{K+1})(1+\cos 2\pi \nu \theta_1)Re(S_{\nu})- \sum_{\nu = 1}^{K_0-1} (1-\frac{\nu}{K+1})(1+\cos 2\pi \nu \theta_1)Re(S_{\nu})\]
  \[ \geq \frac{K+1}{4} -m - 2BK_0. \]
  On the other hand we have
  \[ \sum_{\nu = K_0}^{K}(1-\frac{\nu}{K+1})(1+\cos 2 \pi \nu \theta_1) \leq K.\]
  Therefore, there exist $K_0 \leq \nu \leq K $ such that 
  \[ S_{\nu} \geq \frac{\frac{K+1}{4}-m-2BK_0}{K} > \frac{1}{4}-\frac{m+2BK_0}{K} \geq \frac{1}{4}-\frac{1}{5}=\frac{1}{20}. \]
  Now using $(*)$, we have
  \[ Tr(A^{\nu})=|\lambda_1|^{\nu} \smallskip S_{\nu} \geq (1+\frac{c}{m\log(m)^3})^{K_0} \times \frac{1}{20}  \]
  \[ \geq (1+\frac{cK_0}{m\log(m)^3}) \times \frac{1}{20} >B. \]
 
 \end{proof}
 
 \begin{remark}
 It follows from the proof that conditional on the Schinzel-Zassenhaus conjecture \cite{schinzel1965refinement}, one can replace the upper bound $m^{1+\epsilon}$ for $\nu$, in Proposition \ref{case1}, by a linear bound (with linear constant just depending on $B$).
 \end{remark}
 
 \newtheorem*{conj:Schinzel}{Conjecture \ref{conj:Schinzel}}
 \begin{conj:Schinzel}(Schinzel-Zassenhaus)
There exists a constant $c>0$ such that for any algebraic integer $\lambda$ of degree $d$ which is not a root of unity we have
\[ \overline{|\lambda|} \geq 1+\frac{c}{d}. \]
 \end{conj:Schinzel}
 
 \noindent We need the next Lemma from \cite{montgomery1994ten} for the proof of Proposition \ref{case2}.

\begin{lem}
Let $z_1 , ... , z_n$ be all the roots of a polynomial with real coefficients. Define
\[S_{\nu} = z_1^{\nu}+ ... + z_n^{\nu} \]
Then $S_{\nu} \geq 0$ for some integer $\nu$ in the range $1 \leq \nu \leq n+1$.
\label{lastlemma}
\end{lem}

\begin{proof}
We closely follow the proof from \cite{montgomery1994ten}. Let $\sigma_j$ be the $j$-th elementary symmetric function of $z_1 , ... , z_n$. Therefore, $\sigma_j$ is real for each $1 \leq j \leq n$. Recall the Newton-Girard identities
\[ r \sigma_r = \sum_{\nu=1}^{r}(-1)^{\nu -1}\sigma_{r-\nu} S_{\nu} \] 
for $1 \leq r \leq n$. Suppose that $S_{\nu}<0$ for $1 \leq \nu \leq n$. Using Newton-Girard identities and induction we deduce that $(-1)^j \sigma_j >0$ for $1 \leq j \leq n$. On the other hand, another set of Newton-Girard identities state that 
\[ S_{t+n+1} = \sum_{\nu=t+1}^{t+n}S_{\nu}\medskip(-1)^{t+n-\nu}\sigma_{t+n+1-\nu} \] 
for $t \geq 0$. Putting $t=0$ we see that $S_{\nu}<0$ and $(-1)^{n-\nu}\sigma_{n+1-\nu}<0$ for $1 \leq \nu \leq n$, therefore all summands on the right hand side are positive. Hence $S_{\nu +1}>0$. \\
\end{proof}

\noindent \textbf{Proof of Proposition \ref{case2}}

\begin{proof}
Let $Q(z)$ be the characteristic polynomial of $A$. By the assumption, all roots of $Q$ have absolute value at most one. Recall the following theorem of Kronecker:

\noindent Let $f$ be a monic polynomial with integer coefficients in $z$. If all roots of $f$ have absolute value at most $1$ then $f$ is a product of cyclotomic polynomials and/or a power of $z$.\\
Here, there can not be any power of $z$, since $Q(0)=\det(A)=1$. So we can write $Q$ as
\[ Q(z)=\prod_{j=1}^{l} \Phi_{k_j}(z) \]
where $\Phi_{k_j}(z)$ is the $k_j$-th cyclotomic polynomial and $k_1 \leq k_2 \leq ... \leq k_l$ are natural numbers. In particular, by comparing the degrees we deduce that 
\[ \varphi(k_1)+ ... + \varphi(k_l)=m \]
where $\varphi$ is the Euler totient function. Take $B'$ such that for each $t>B'$ we have $\varphi(t)>B$. This is possible since $\lim _{t \rightarrow \infty} \varphi (t)=\infty$. In fact more is true. For any $\delta >0$ we have (see \cite{hardy1979introduction} Theorem 327) 

\[ \lim_{t \rightarrow \infty} \frac{\varphi(t)}{t^{1-\delta}}=\infty. \]

\noindent Firstly, we specify how large $m$ should be. We require that $m > (B')!$. Moreover assume that $m$ is large enough so that for each $t \geq m^{1+\epsilon}$ we have $\varphi(t) > m+B$. Note that $k_l < m^{1+ \epsilon}$ since $\varphi(k_l) \leq m$. We consider two cases\\

\noindent 1) $B' < k_l < m^{1+\epsilon}$. Let $g(z)$ be the polynomial whose roots are the $k_l$-powers of the roots of $\frac{Q(z)}{\Phi_{k_l}(z)}$ allowing repetitions. Hence, $g$ has integer coefficients by Newton's identities and $\deg(g) < m$. Let $\nu = k_l \cdot \nu' $ where $1 \leq \nu ' \leq m$ is chosen such that the sum, $S$, of the $\nu'$-powers of the roots of $g$ is non-negative. Such a $\nu '$ exists by Lemma \ref{lastlemma}.

\noindent Since $\varphi (k_l) > B$ we have the following
\[ Tr(A^{\nu}) = \varphi(k_l)+ S > B. \]
\noindent Note that $\nu = k_l \cdot \nu'$ is at most $ m^{1+\epsilon} \cdot m = m^{2+\epsilon}. $\\

\noindent 2) $k_l \leq B'$. Take $\nu = (B')!$. Then 
\[ Tr(A^{\nu}) = m > B. \]
 This completes the proof.
\end{proof}

\begin{conj}
There is a constant $C>0$ such that for all $g \geq 2$ and $n \geq 0$ we have the following:
\[ \l_{g,n} \geq \frac{C}{g} \medskip \frac{\log(|\chi(S)|)}{|\chi(S)|}.  \]
\end{conj}
\noindent It seems plausible to prove the above conjecture by improving Propositions \ref{case1} and \ref{case2}. As mentioned previously in a remark, the bound in Proposition \ref{case1} can be replaced by a linear bound conditional on the Schinzel-Zassenhaus conjecture. We expect a similar linear bound to be true in Proposition \ref{case2}, however we do not know how to prove it.

\appendix

\section{Extending the upper bound}
Recall that Tsai has proved the following upper bound for $l_{g,n}$: There is a constant $C>0$ such that for any $g \geq 2$ and any $n \geq 12g+7$, we have (see \cite{tsai2009asymptotic}, page 2271): 
\[ \l_{g,n} \leq C \medskip g \smallskip \frac{\log|\chi(S)|}{|\chi(S)|}. \]
In this appendix, we extend this result to the whole range $g \geq 2$ and $n\geq 0$. It is enough to prove a similar result for $n \leq 12g+6$. The proof in this case is straightforward and we bring it for the sake of completeness.\\ 
For functions $f$ and $g$, we use the notation $f \asymp g$ when there exists a positive constant $K$ such that:
\[ \frac{1}{K} \smallskip f < g < K \smallskip f. \]
Therefore in the range $n \leq 12g+6$, we have:
\[ |\chi(S)| \asymp g \]
Hence, the desired upper bound is equivalent to the following: There exist a contact $M$ such that 
\[ \l_{g,n} \leq M \smallskip \log(g) \]
To prove the last inequality, it is enough to find a constant $M>0$, independent of $g$ and $n$, such that there always exsits a pseudo-Anosov map $f : S_{g,n} \longrightarrow S_{g,n}$ with $\lambda(f) \leq g^M$.\\ 
Consider the following set of Penner curves $\mathcal{A}= \{ a_1 , \dots , a_k \}$ (red curves) and $\mathcal{B}= \{ b_1 , \dots , b_{\ell} \}$ (blue curves) where $k = g+1$ and $\ell = g+n-1$ (Figure \ref{upperboundexample}). Let $\tau_r$ be the composition of positive Dehn twists along red curves and $\tau_b$ be the composition of positive Dehn twists along blue curves. Note that the order of composition in $\tau_{r}$ (respectively $\tau_b$) is not important since they commute. Define 
\[ f := \tau_r \circ (\tau_b)^{-1}. \] 
Since $\mathcal{A} \cup \mathcal{B}$ fills the surface and each complementary region is a disk or a once punctured disk, therefore the conditions of Penner's construction are satisfied. By Penner's construction, the map $f$ is pseudo-Anosov and a train track $\tau$ can be obtained from the union $\mathcal{A} \cup \mathcal{B}$ by smoothing the intersection points in the proper way. Let $V$ be the vector space spanned by transverse measures on $\tau$. Define $H$ as the linear subspace of $V$ spanned by transverse measures supported on only one of the curves in $\mathcal{A} \cup \mathcal{B}$. The vector space $H$ is invariant under the action of $f$ and the stretch factor of $f$ is equal to the spectral radius of the induced action on $H$. The induced action on $H$ can be represented by a $(k+\ell) \times (k+\ell)$ non-negative integral matrix $A$. By our specific construction of $f$, each entry of $A$ is bounded above by $\max \{3n+4,7 \} $. Since the spectral radius of a non-negative matrix is bounded above by the maximum row sum, we have: 
\[ \lambda(f) \leq (k+\ell) \times \max \{ 3n+4 , 7 \} \leq (14g+6)(36g+22) \leq g^M \]
for some $M>>0$.

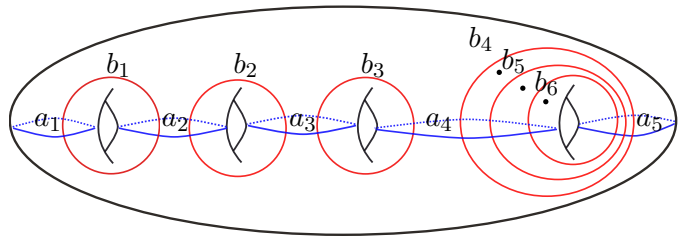
\begin{figure}
\labellist
\pinlabel $a_1$ at 18 52
\pinlabel $a_2$ at 76 52
\pinlabel $a_3$ at 134 52
\pinlabel $a_4$ at 196 52
\pinlabel $a_5$ at 292 52
\pinlabel $b_1$ at 50 78
\pinlabel $b_2$ at 108 78
\pinlabel $b_3$ at 166 78
\pinlabel $b_4$ at 215 90
\pinlabel $b_5$ at 230 80
\pinlabel $b_6$ at 245 70
\endlabellist
\centering 
\includegraphics[width= 3.5 in]{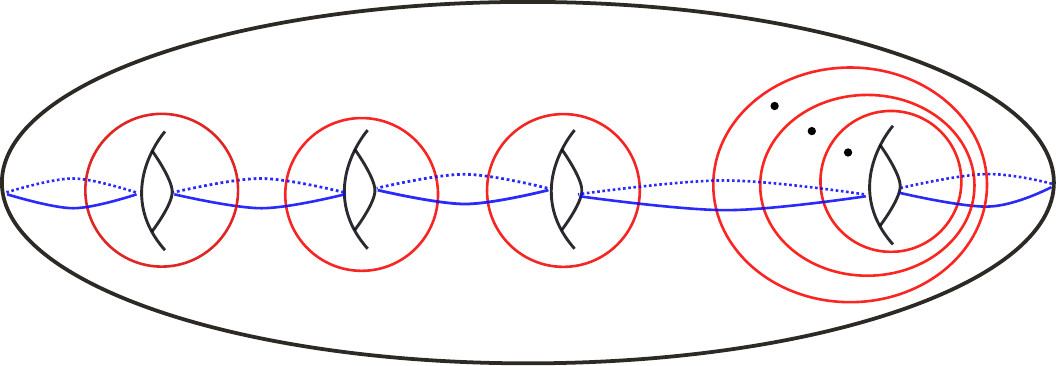}
\caption{Penner curves for the map $f : S_{4,3} \longrightarrow S_{4,3}$}
\label{upperboundexample}
\end{figure}

\bibliographystyle{plain}
\bibliography{references}

\end{document}